\documentclass[12pt]{amsart}
\usepackage{amsthm} 

\usepackage{subfigure}
\usepackage{color}
\usepackage{amsmath}
\usepackage{amsthm}
\usepackage{amstext}
\usepackage{amssymb}
\usepackage{amsfonts}
\usepackage{graphicx}
\usepackage{young}
\usepackage{multicol}
\usepackage{mathrsfs}
\usepackage{stmaryrd}
\usepackage{bbm}
\usepackage[all]{xy}
\usepackage{mathabx}
\usepackage{tikz}
\usepackage{mathtools}
\usepackage{tabularx}
\usepackage{array}
\usepackage{commath}

\usepackage{caption}
\usepackage[normalem]{ulem}

\definecolor{darkblue}{rgb}{0.0,0,0.7}

\newcommand{\red}{\color{red}}
\newcommand{\newword}[1]{\textcolor{darkblue}{\textbf{\emph{#1}}}}

\theoremstyle{plain}
\theoremstyle{definition}
\newtheorem{theorem}{Theorem}[section]
\newtheorem{conjecture}[theorem]{Conjecture}

\newtheorem{lemma}[theorem]{Lemma}
\newtheorem{definition}[theorem]{Definition}

\newtheorem{proposition}[theorem]{Proposition}

\DeclareMathAlphabet{\mathpzc}{OT1}{pzc}{m}{it}

\newcommand{\Z}{\mathbb{Z}}
\newcommand{\id}{\mathrm{id}}
\newcommand{\xdownarrow}[1]{%
  {\left\downarrow\vbox to #1{}\right.\kern-\nulldelimiterspace}
}

\usepackage{amssymb,amsmath,tabularx,graphicx}
\usepackage{tikz}
\usetikzlibrary[calc,intersections,through,backgrounds,arrows,decorations.pathmorphing,math]

\begin{document}

\title{Path-cordial abelian groups}
\date{\today}

\author{Rebecca Patrias}
\address[RP]{\parbox{\linewidth}{Department of Mathematics, University of St.\ Thomas, St.\ Paul, MN 55105}}
\email{\parbox[t]{\linewidth}{patr0028@stthomas.edu}}
\author{Oliver Pechenik}
\address[OP]{\parbox{\linewidth}{Department of Mathematics, University of Michigan, Ann Arbor, MI 48109}}
\email{\parbox[t]{\linewidth}{pechenik@umich.edu}}



\begin{abstract}
A labeling of the vertices of a graph by elements of any abelian group $A$ induces a labeling of the edges by summing the labels of their endpoints. Hovey defined the graph $G$ to be \emph{$A$-cordial} if it has such a labeling where the vertex labels and the edge labels are both evenly-distributed over $A$ in a technical sense. His conjecture that all trees $T$ are $A$-cordial for all cyclic groups $A$ remains wide open, despite significant attention. Curiously, there has been very little study of whether Hovey's conjecture might extend beyond the class of cyclic groups.

We initiate this study by analyzing the larger class of finite abelian groups $A$ such that all path graphs are $A$-cordial. We conjecture a complete characterization of such groups, and establish this conjecture for various infinite families of groups as well as for all groups of small order.
\end{abstract}

\maketitle
\section{Introduction}\label{sec:intro}
Let $A$ be a finite abelian group of order $n$. A labeling of the vertices of any graph $G$ by elements of $A$ induces a labeling of the edges of $G$ by associating to each edge the sum of the labels on its endpoints. Hence, every such vertex labeling gives rise to two integer partitions, each with at most $n$ parts: $\lambda^v \coloneqq (\lambda^v_1 \geq \dots \geq \lambda^v_n \geq 0)$ recording the number of vertices with each label and $\lambda^e \coloneqq (\lambda^e_1 \geq \dots\geq \lambda^e_n \geq 0)$ recording the number of edges with each label (both arranged in decreasing order). We say a partition $\lambda = (\lambda_1 \geq \lambda_2 \geq \cdots)$ is \newword{almost rectangular} if, for all $i$, we have $\lambda_i \in \{\lambda_1, \lambda_1-1, 0\}$. 

In early work, Mark Hovey \cite{Hovey} introduced the following notion.

\begin{definition}
A graph $G$ is \newword{$A$-cordial} if there is a vertex labeling of $G$ such that both partitions $\lambda^v, \lambda^e$ are almost rectangular. 
\end{definition}

Hovey's definition subsumes other famous notions in graph labeling. For example, the graph $G$ is $\Z_2$-cordial if and only if it is \emph{cordial} in the sense of I.~Cahit \cite{Cahit}, while $G$ is $\Z_{|E(\Gamma)|}$-cordial if and only if it is \emph{harmonious} in the sense of R.L.~Graham and N.J.A.~Sloane \cite{Graham.Sloane}. (Here $\Z_k$ denotes the cyclic group of order $k$, written additively.) For an extensive survey of graph labeling, see \cite{Gallian}.

Most work to date has taken the form of fixing a group $A$ (nearly always cyclic) and asking which graphs $G$ are $A$-cordial. Results along these lines, while numerous, are mostly piecemeal and having {\it ad hoc} proofs. Here, we consider a dual problem. 

\begin{definition}
Let $\mathbb{G}$ be a family of graphs. We say a group $A$ is \newword{$\mathbb{G}$-cordial} if every $G \in \mathbb{G}$ is $A$-cordial. We say $A$ is \newword{weakly $\mathbb{G}$-cordial} if all but finitely many $G \in \mathbb{G}$ are $A$-cordial.
\end{definition}

A major open problem is Hovey's conjecture \cite[Conjecture~1]{Hovey} that all trees are $A$-cordial for all cyclic groups $A$. In other language, we have the following.

\begin{conjecture}[\cite{Hovey}]\label{conj:tree}
	Let $\mathbb{T}$ be the class of trees. Then $\mathbb{Z}_k$ is $\mathbb{T}$-cordial for all $k$.
\end{conjecture}
Currently, the only nontrivial finite groups known to be $\mathbb{T}$-cordial are the small cyclic groups $\mathbb{Z}_2$ \cite{Cahit}, $\Z_3$ \cite{Hovey}, $\Z_4$ \cite{Hovey}, $\Z_5$ \cite{Hovey}, $\mathbb{Z}_6$ \cite{Driscoll.Krop.Nguyen}, and $\Z_7$ \cite{Driscoll}. We wish to gain insight into Conjecture~\ref{conj:tree} by considering the class of $\mathbb{T}$-cordial groups more broadly. Certainly, not all abelian groups are $\mathbb{T}$-cordial, as it is easy to check that the four-vertex path $P_4$ is not $\Z_2 \times \Z_2$-cordial. However, it is currently unknown whether any finite noncyclic group is $\mathbb{T}$-cordial; indeed, there are not even any conjectures in this direction. The results of this paper, however, may be taken as evidence that many finite noncyclic groups are also $\mathbb{T}$-cordial. Hence, the appropriate level of generality for studying Conjecture~\ref{conj:tree} may be broader than the class of cyclic groups.

It seems reasonable to start this program by considering the broader family of $\mathbb{P}$-cordial groups, where $\mathbb{P}$ denotes the class of \emph{path graphs} (i.e., trees with no vertex of degree $>2$). Hovey \cite[Theorem~2]{Hovey} showed that all cyclic groups are $\mathbb{P}$-cordial, whereas $\Z_2 \times \Z_2$ is not. To our knowledge, no other results about $\mathbb{P}$-cordiality have been established to date. We make the following conjecture.

\begin{conjecture}\label{conj:path}
	 A finite abelian group $A$ is $\mathbb{P}$-cordial if and only if it is not a nontrivial product of copies of $\Z_2$ (equivalently, if and only if there exists $a \in A$ with $|a|>2$).
\end{conjecture}

Our main results are evidence toward Conjecture~\ref{conj:path}. The following proves one direction.

\begin{theorem}\label{thm:exp2}
	If $A = \bigtimes_{i=1}^m \Z_2$ is a product of copies of $\Z_2$ ($m >1$), then $P_{2m}$ and $P_{2m+1}$ are not $A$-cordial (and so $A$ is not $\mathbb{P}$-cordial).
\end{theorem}

For the other direction of Conjecture~\ref{conj:path}, our main result is the following.

\begin{theorem}\label{thm:odd}
	If $|A|$ is odd, then $A$ is $\mathbb{P}$-cordial.
\end{theorem}

It remains to understand the case of groups $A$ that are of even order and have an element of order greater than $2$.  We give some partial results in that setting. We also verify Conjecture~\ref{conj:path} for all abelian groups $A$ with $|A| < 24$.

While $\bigtimes_{i=1}^2 \Z_2$ is not $\mathbb{P}$-cordial, it is known to be \emph{weakly} $\mathbb{P}$-cordial \cite[Theorem~3.4]{Pechenik.Wise}. We show that $\bigtimes_{i=1}^3 \Z_2$ is similarly weakly $\mathbb{P}$-cordial, while not  $\mathbb{P}$-cordial. On the basis of these examples, \cite[Corollary~4.3]{Pechenik.Wise} (showing that there are infinitely-many $A$-cordial paths for each $A$), and Conjecture~\ref{conj:path}, we also expect the following to hold.

\begin{conjecture}\label{conj:weakpath}
	All finite abelian groups are weakly $\mathbb{P}$-cordial.
	\end{conjecture}

This paper is structured as follows. In Section~\ref{sec:general}, we prove some results that hold for all finite abelian groups $A$ and that we will need in later sections. In particular, Theorem~\ref{thm:singlecheck} allows us to demonstrate $\mathbb{P}$-cordiality of $A$ by showing the $A$-cordiality of a single path. In Section~\ref{sec:odd}, we apply Theorem~\ref{thm:singlecheck} to prove Theorem~\ref{thm:odd}, showing that all $A$ of odd order are $\mathbb{P}$-cordial. In Section~\ref{sec:order2}, we study products of groups of order two, proving Theorem~\ref{thm:exp2} and verifying that $\bigtimes_{i=1}^3 \Z_2$ is weakly $\mathbb{P}$-cordial. Finally, Section~\ref{sec:othereven} is devoted to other groups of even order, including a complete analysis of $\mathbb{P}$-cordiality for all abelian groups of small order.

\section{Results for general groups}\label{sec:general}

Here we collect various results that hold for all finite abelian groups $A$. We use these results later to establish our main results.

\begin{lemma}\label{lem:extby1}
Let $A$ be any abelian group of order $n$. Let $f,k \in \Z_{\geq 0}$ be nonnegative integers such that $f \leq \frac{n}{2}$. Then if the path $P_{nk+f}$ is $A$-cordial, so is $P_{nk+f+1}$.
\end{lemma}
\begin{proof}
This is a generalization of \cite[Lemma~3]{Hovey} and the proof is essentially the same, but we include the details for completeness.

Consider an $A$-cordial labeling of $P_{nk+f}$ with almost rectangular partitions $\lambda^v, \lambda^e$.  Attach a new vertex $x$ to one of the ends of the path by a new edge $y$. It suffices to exhibit an appropriate vertex label for $x$.
There are $n-f$ available vertex labels for it that would keep $\lambda^v$ almost rectangular.

If $f=0$, then there is a unique label for $y$ that makes $\lambda^e$ almost rectangular, and we may choose the label for $x$ accordingly.

If $f>0$, there are only $f-1$ edge labels on $y$ that would make $\lambda^e$ no longer almost rectangular. So we can find an appropriate label for $x$ provided that $f-1 < n-f$, that is if $2f < n+1$.
\end{proof}

\begin{lemma}\label{lem:shift}
Suppose $G$ is an $A$-cordial graph and $a \in A$. Then for any $A$-cordial labeling of $G$, we may add $a$ to each vertex label to obtain another $A$-cordial labeling. 
\end{lemma}
\begin{proof}
	This is \cite[Lemma~1]{Hovey}.
\end{proof}

\begin{lemma}\label{lem:gluing}
	Suppose $|A|=n$ and let $k,m \in \Z_{> 0}$ be positive integers. If $P_k$ and $P_{mn}$ are both $A$-cordial, then so is $P_{mn+k}$. 
\end{lemma}
\begin{proof}
Consider $A$-cordial labelings of $P_k$ and $P_{mn}$. Every element of $A$ appears exactly $m$ times as a vertex label of $P_{mn}$. There is a unique element $a \in A$ that appears $m-1$ times as an edge label of $P_{mn}$, while all others appear exactly $m$ times. By Lemma~\ref{lem:shift}, we may assume that $P_{mn}$ has an endpoint labeled by the identity element $\id$ and that $P_k$ has an endpoint labeled $a$. Join the $\id$-end of $P_{mn}$ to the $a$-end of $P_k$ by a new edge (necessarily labeled $a$) to obtain a labeled $P_{mn+k}$.
	
	Then we have added to $P_k$ exactly $m$ of each vertex label and exactly $m$ of each edge label. That is, for all $1 \leq i \leq n$, we have $\lambda^v_i(P_{mn+k}) = \lambda^v_i(P_k) + m$ and $\lambda^e_i(P_{mn+k}) = \lambda^e_i(P_k) + m$. Since the labeling of $P_k$ was $A$-cordial, it follows that this labeling of $P_{mn+k}$ is as well.
\end{proof}

\begin{theorem}\label{thm:singlecheck}
Suppose $|A|=n$. Then $A$ is $\mathbb{P}$-cordial if and only if $P_n$ is $A$-cordial.
\end{theorem}
\begin{proof}
	One implication is trivial. For the other, suppose $P_n$ is $A$-cordial. Then iteratively deleting vertices from either end of an $A$-cordial labeling of $P_n$, we see that $P_k$ is $A$-cordial for all $k < n$.
	
	Now consider any path $P_m$. Write $m = hn + k$ for some $h,k \in \Z$ and $0 \leq k < n$. Then we may construct an $A$-cordial labeling of $P_m$ by gluing $h$ $A$-cordial labelings of $P_n$ to an $A$-cordial labeling of $P_k$, as in Lemma~\ref{lem:gluing}.
\end{proof}

The previous theorem makes it easy to demonstrate that a group $A$ is $\mathbb{P}$-cordial because it is thereby sufficient to exhibit a single $A$-cordial labeling of a single graph. For any fixed $A$, Theorem~\ref{thm:singlecheck} indeed makes it a finite check to determine whether $A$ is $\mathbb{P}$-cordial. In practice, however, this exhaustive check is not feasible for groups $A$ that are not of extremely small order.

\section{Groups of odd order}\label{sec:odd}

In this section, we prove Theorem~\ref{thm:odd}, showing that all groups of odd order are $\mathbb{P}$-cordial. The reader may find the technical details of the proof clarified by consulting the example for $\Z_3 \times \Z_3$ 
displayed in Figure~\ref{fig:Z3xZ3}.

\begin{lemma}\label{lem:puffingcycle}
	Suppose $|A| = n$ and the cycle $C_n$ is $A$-cordial. Then, for $k$ odd, the cycle $C_{kn}$ is $A \times \Z_k$-cordial.
\end{lemma}
\begin{proof}
	Consider an $A$-cordial labeling of $C_n$ and write $\ell(v)$ for the label of vertex $v$. Fix a cyclic orientation of $C_n$. Subdivide each side of $C_n$ by inserting $k-1$ new vertices into each edge of $C_n$ to obtain a cycle $C_{kn}$. Each (directed) edge $q = \overrightarrow{xy}$ of $C_n$ with endpoints $x,y$ becomes an induced (oriented) path $P_n$ with vertices $x \rightarrow q_1 \rightarrow \dots \rightarrow q_{k-1} \rightarrow y$. Label these vertices alternately $\ell(x)$ and $\ell(y)$, so that $x$ is still labeled $\ell(x)$ and $y$ is still labeled $\ell(y)$. More precisely, we define $\ell(q_i)=\ell(x)$ if $i$ is even and $\ell(q_i)=\ell(y)$ if $i$ is odd.	
	
	We claim that this process results in an $A$-cordial labeling of $C_{kn}$. Note that if $x$ has neighbors $y,z$ in $C_n$, then in our $C_{kn}$ the label $\ell(x)$ appears  on $\frac{k-1}{2}+1$ vertices along the path from $x$ to $y$ and on $\frac{k-1}{2}+1$ vertices along the path from $z$ to $x$, with the vertex $x$ being shared. Hence, exactly $k$ vertices of $C_{kn}$ are labeled $\ell(x)$. Since we started with a labeling of $C_n$ with every element of $A$ appearing exactly once as a vertex label, this means that the vertex labels of our $C_{kn}$ yield an almost rectangular partition.
	
	All of the edges along the path from $x$ to $y$ in $C_{kn}$ are labeled $\ell(x) + \ell(y)$. Hence the edge labels of our $C_{kn}$ are the edge labels of the original $C_n$ in the same order, but each now repeated $k$ times in a row. Since each element of $A$ appeared exactly once as an edge label of $C_n$, this means that the edge labels of the $C_{kn}$ also give an almost rectangular partition, so our construction is an $A$-cordial labeling of $C_{kn}$.

	Finally, we will convert this labeling of $C_{kn}$ into an $A \times \Z_k$-cordial labeling. For each of the vertices $x$ from $C_n$, label $x$ by the ordered pair $(\ell(x), 0) \in A \times \Z_k$. Along the oriented path $x \rightarrow q_1 \rightarrow \dots \rightarrow q_{k-1} \rightarrow y$, label each vertex $q_{2i}$ by $(\ell(x), i) \in A \times \Z_k$ and each vertex $q_{2i+1}$ by $(\ell(y), i + \frac{k-1}{2} ) \in A \times \Z_k$.
	
	It is straightforward to see that each element of $A \times \Z_k$ appears now as a vertex label exactly once. Moreover, the second coordinates of the edge labels rotate cyclically through the elements of $\Z_k$ in cyclic order. Since each element of $A$ appears as a first coordinate on exactly $k$ consecutive edges, this implies that each element of $A \times \Z_k$ also appears as an edge label exactly once, so we have constructed an $A \times \Z_k$-cordial labeling of $C_{nk}$.
\end{proof}

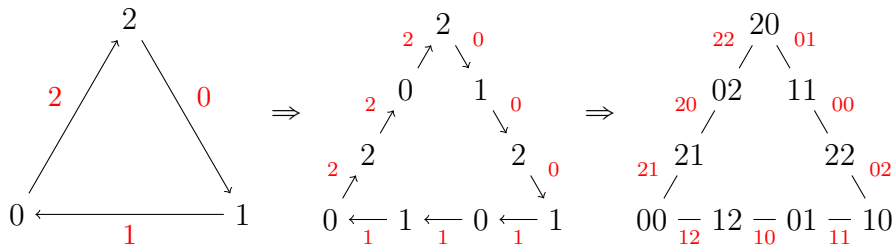
\begin{figure}[h]
    \centering
        \begin{tikzpicture}
    \tikzmath{\x = 0.333333333333; \y =1.73205080757;} 
 \node (x) at (0,0) {$0$};
 \node (y) at (3,0) {$1$};
 \node (z) at (1.5, 1.5 * \y) {$2$};
 \draw[<-] (x) --  (y) node [midway,below] {\small{$\red{1}$}};
 \draw[->] (x) -- (z) node [midway,above left] {\small{$\red{2}$}};
 \draw[->] (z) -- (y) node [midway,above right] {\small{$\red{0}$}};
     \end{tikzpicture}
     \vspace{.1in}
     \raisebox{.7in}{$\Rightarrow$}
     \vspace{.1in}
  \begin{tikzpicture}
    \tikzmath{\x = 0.333333333333; \y =1.73205080757;} 
 \node (x) at (0,0) {$0$};
 \node (q1) at (1,0) {$1$};
 \node (q2) at (2,0) {$0$};
 \node (y) at (3,0) {$1$};
 \node (p1) at (0.5, 0.5 * \y) {$2$};
 \node (p2) at (1.0, 1.0 * \y) {$0$};
 \node (z) at (1.5, 1.5 * \y) {$2$};
 \node (r1) at (2.0, 1.0 * \y) {$1$};
 \node (r2) at (2.5, 0.5 * \y) {$2$};
 \draw[<-] (x) -- (q1) node [midway,below] {\tiny{$\red{1}$}};
 \draw[<-] (q1) -- (q2) node [midway,below] {\tiny{$\red{1}$}};
 \draw[<-] (q2) -- (y) node [midway,below] {\tiny{$\red{1}$}};
 \draw[->] (x) -- (p1) node [midway,above left] {\tiny{$\red{2}$}};
 \draw[->] (p1) -- (p2) node [midway,above left] {\tiny{$\red{2}$}};
 \draw[->] (p2) -- (z) node [midway,above left] {\tiny{$\red{2}$}};
 \draw[->] (z) -- (r1) node [midway,above right] {\tiny{$\red{0}$}};
 \draw[->] (r1) -- (r2) node [midway,above right] {\tiny{$\red{0}$}};
 \draw[->] (r2) -- (y) node [midway,above right] {\tiny{$\red{0}$}};
     \end{tikzpicture}
     \vspace{.1in}
     \raisebox{.7in}{$\Rightarrow$}
     \vspace{.1in}
  \begin{tikzpicture}
    \tikzmath{\x = 0.333333333333; \y =1.73205080757;} 
 \node (x) at (0,0) {$00$};
 \node (q1) at (1,0) {$12$};
 \node (q2) at (2,0) {$01$};
 \node (y) at (3,0) {$10$};
 \node (p1) at (0.5, 0.5 * \y) {$21$};
 \node (p2) at (1.0, 1.0 * \y) {$02$};
 \node (z) at (1.5, 1.5 * \y) {$20$};
 \node (r1) at (2.0, 1.0 * \y) {$11$};
 \node (r2) at (2.5, 0.5 * \y) {$22$};
 \draw (x) -- (q1) node [midway,below] {\tiny{$\red{12}$}};
 \draw (q1) -- (q2) node [midway,below] {\tiny{$\red{10}$}};
 \draw (q2) -- (y) node [midway,below] {\tiny{$\red{11}$}};
 \draw (x) -- (p1) node [midway,above left] {\tiny{$\red{21}$}};
 \draw (p1) -- (p2) node [midway,above left] {\tiny{$\red{20}$}};
 \draw (p2) -- (z) node [midway,above left] {\tiny{$\red{22}$}};
 \draw (z) -- (r1) node [midway,above right] {\tiny{$\red{01}$}};
 \draw (r1) -- (r2) node [midway,above right] {\tiny{$\red{00}$}};
 \draw (r2) -- (y) node [midway,above right] {\tiny{$\red{02}$}};
     \end{tikzpicture}

    \caption{Example of constructing a $\Z_3\times\Z_3$-cordial labeling of $C_9$ from a $\Z_3$-cordial labeling of $C_3$, as described in the proof of Lemma~\ref{lem:puffingcycle}.}
    \label{fig:Z3xZ3}
\end{figure}

\begin{lemma}\label{lem:ordercycle}
	Suppose $|A| = n$ is odd. Then $C_n$ is $A$-cordial.
\end{lemma}
\begin{proof}
By induction on $n$. If $A$ is a cyclic group, this is \cite[Theorem~9]{Hovey}. Otherwise, write $A$ as $B \times \Z_k$ for some odd $k$. Then $|B| = b$ is also odd, so by induction, $C_b$ is $B$-cordial. Therefore Lemma~\ref{lem:puffingcycle} gives that $C_{kb}$ is $B \times \Z_k$-cordial, i.e., $C_n$ is $A$-cordial.
\end{proof}

\begin{proof}[Proof of Theorem~\ref{thm:odd}]
Let $|A| = n$ be odd. Then by Lemma~\ref{lem:ordercycle}, the cycle $C_n$ is $A$-cordial. Take an $A$-cordial labeling of $C_n$. Deleting any edge yields an $A$-cordial labeling of the path $P_n$. Since $P_n$ is $A$-cordial, Theorem~\ref{thm:singlecheck} then implies that $A$ is $\mathbb{P}$-cordial, as desired.
\end{proof}

\section{Products of order-$2$ groups}\label{sec:order2}
 
 In this section, we assume $A = \bigtimes_{i=1}^m \Z_2$ for some $m>1$. We first prove Theorem~\ref{thm:exp2}, showing that such an $A$ cannot be $\mathbb{P}$-cordial.

\begin{proof}[Proof of Theorem~\ref{thm:exp2}]
Suppose we had an $A$-cordial labeling of $P_{2m}$. Then each $a \in A$ appears exactly once as a vertex label, and exactly one element of $A$ does not appear as an edge label. This missing element must be $\id$, since an edge labelled $\id$ can only come from adjacent vertices labeled $a$ and $-a$, but every element of $A$ is its own inverse. Therefore, the edge labels are the $2m-1$ nonidentity elements of $A$, each appearing exactly once. 

It is easy to see that the sum of all the (nonidentity) elements of $A$ is $\id$. Hence, the sum of the edge labels is $\id$. On the other hand, the edge labels come from adding adjacent vertex labels. Thus, the sum of the edge labels is the sum of the two leaf vertex labels plus twice the labels of all internal vertices. However, twice any group element is the identity, so the sum of the edge labels is just the sum of the two leaf vertex labels. Hence, the labels on the leaf vertices are each other's inverses, contradicting that they must be distinct. Therefore, $P_{2m}$ is not $A$-cordial.

Now, suppose we had an $A$-cordial labeling of $P_{2m+1}$. Then each $a\in A$ appears exactly once as an edge label. Let $e$ be the edge labeled id, and let its endpoints be $x,y$. Then the labels on $x$ and $y$ sum to id and hence are equal. Contract the edge $e$ to get a labeled $P_{2m}$. It has each $a\in A$ appearing exactly once as a vertex label and each nonidentity element appearing exactly once as an edge label, contradicting that $P_{2m}$ is not $A$-cordial.
\end{proof}

The special case $m=2$ was studied in \cite{Pechenik.Wise}. By Theorem~\ref{thm:exp2}, $P_4$ and $P_5$ are not $\Z_2 \times \Z_2$-cordial. However, \cite[Theorem~3.4]{Pechenik.Wise} shows that all other paths are $\Z_2 \times \Z_2$-cordial, so that $\Z_2 \times \Z_2$ is \emph{weakly} $\mathbb{P}$-cordial. As evidence towards Conjecture~\ref{conj:weakpath}, we here establish analogous results in the case $m=3$.

For notational convenience, let $M=\Z_2\times\Z_2\times \Z_2$. We further identify the elements of $M$ with binary strings of length $3$, added componentwise.

\begin{proposition}
	The group $M$ is weakly $\mathbb{P}$-cordial. More precisely, all paths are $M$-cordial except $P_8$ and $P_9$.
\end{proposition}
\begin{proof}
	Theorem~\ref{thm:exp2} gives that $P_8$ and $P_9$ are not $M$-cordial.	
	
	 Lemma~\ref{lem:extby1} shows that $P_k$ is $M$-cordial for $k \leq 5$. 	The reader may confirm that 
	 \[
	 100-000-001-010-111-101
	 \]
	 is an $M$-cordial labeling of $P_6$ (here we record only the vertex labels), while
	 \[
	 100-000-001-010-111-101-011
	 \]
	 is an $M$-cordial labeling of $P_7$.
	 
	Additionally, 
\[000-111-001-010-110-011-011-100-101-111\]
 is an $M$-cordial labeling of $P_{10}$. Therefore, the $M$-cordiality of $P_{11}$, $P_{12}$, and $P_{13}$ follows immediately by Lemma~\ref{lem:extby1}.

 Further observe that the following are $M$-cordial labelings of $P_{14}$, $P_{15}$, and $P_{16}$:
{\small
\begin{align*}
&P_{14}: &000-111-001-010-110-011-011-100-101-111-001-010-110-100 \\
&P_{15}: &000-111-001-010-110-011-011-100-101-111-001-010-110-100-101 \\
&P_{16}: &000-000-111-001-010-110-011-011-100-101-111-001-010-110-100-101
\end{align*}}

Since $16 = 2 \cdot |M|$, if $k = 16m + j$ for some $0 \leq j < 16$ with $j \notin \{8,9 \}$, then we may construct an $M$-cordial labeling of $P_k$ by Lemma~\ref{lem:gluing}, gluing an $M$-cordial labeling of $P_j$ (as constructed above) to $m$ copies of the $M$-cordial labeling of $P_{16}$ given above. Hence, we have shown that $P_k$ is $M$-cordial for all $k$ not congruent to $8$ or $9$ modulo $16$. 

A $M$-cordial labeling of $P_{24}$ is {\tiny 
\[
000-001-101-000-111-001-010-110-011-011-100-101-111-001-010-110-100-100-011-010-111-101-110-000.
\]} 
By Lemma~\ref{lem:extby1}, it then follows that $P_{25}$ is also $M$-cordial.  Using Lemma~\ref{lem:gluing} to attach an appropriate number of labeled copies of $P_{16}$, we obtain $M$-cordial labelings of all remaining paths.
\end{proof}

\section{Other groups of even order}\label{sec:othereven}

While we know from Section~\ref{sec:odd} that all groups of odd order are $\mathbb{P}$-cordial and from Section~\ref{sec:order2} that nontrivial products of order-$2$ groups are not, the situation for other groups of even order is more mysterious. Although Conjecture~\ref{conj:path} predicts that all such groups should also be $\mathbb{P}$-cordial, we only establish limited progress in this direction, as well as verifying Conjecture~\ref{conj:path} for groups of small order.

\begin{proposition} \label{prop:smallorder}
If $A$ is an abelian group with $|A| < 24$ that is not a nontrivial direct product of groups of order $2$, then $A$ is $\mathbb{P}$-cordial.
\end{proposition}
\begin{proof} Let $n = |A|$ and assume $n < 24$. If $A$ is cyclic, then we are done by \cite[Theorem~2]{Hovey}. If $n$ is odd, we are done by Theorem~\ref{thm:odd}.

Up to group isomorphism, it remains to consider the following groups $A$:
\begin{enumerate}
    \item[(a)] $\Z_2\times \Z_4$,
    \item[(b)]  $\Z_2\times\Z_6$,
    \item[(c)] $\Z_2\times \Z_8$,
    \item[(d)] $\Z_4\times \Z_4$,
    \item[(e)] $\Z_3 \times \Z_6$, 
    \item[(f)] $\Z_2 \times \Z_{10}$.
\end{enumerate}
By Theorem~\ref{thm:singlecheck}, it then suffices to exhibit an $A$-cordial labeling of $P_n$ for each of these six groups. For compactness, we write $ij$ to denote the group element $(i,j) \in \Z_a \times \Z_b$, and record only the vertex labels. The reader may verify directly that the following are $A$-cordial labelings for the corresponding groups:
\begin{enumerate}
    \item[(a)] $ 00-12-10-01-02-03-11-13$,
    \item[(b)] $03 - 00 - 15 - 13 - 11 - 05 - 02 - 12 - 14 - 04 - 01 - 10$,
    \item[(c)] $00-14-16-07-04-15-12-13-11-02-10-06-03-05-01-17$,
    \item[(d)] $00-20-23-03-12-11-33-01-13-32-21-10-22-30-31-02$,
    \item[(e)] $00-25-21-01-02-22-12-15-04-11-24-20-10-13-05-03-23-14$,
    \item[(f)] $00-11-07-05-12-15-14-06-18-16-09-01-04-19-17-02-10-13-03-08$.
\end{enumerate}
\end{proof}

The proof of Proposition~\ref{prop:smallorder} relies on ad hoc construction of some examples of $A$-cordial labelings.  For even-order groups $A$ that are generated by two elements, the following proposition gives a uniform construction of $A$-cordial paths. It seems like this construction (in combination with the lemmas of Section~\ref{sec:general}) would be useful in a proof that such groups are at least \emph{weakly} $\mathbb{P}$-cordial, a step towards establishing Conjectures~\ref{conj:path} and~\ref{conj:weakpath}.

\begin{proposition}\label{prop:doublepath}
If $A = \mathbb{Z}_2 \times \mathbb{Z}_k$ and $n = |A|,$ then $P_{2n}$ is $A$-cordial.
    \end{proposition}
\begin{proof}
If $k$ is odd, then $A$ is a cyclic group, so all paths are $A$-cordial by \cite[Theorem~2]{Hovey}.

Hence assume $k=2m$ is even. We will exhibit an explicit $A$-cordial labeling of $P_{2n}$ by first building an auxiliary labeling of $P_{2n}$ that is not $A$-cordial and then slightly modifying this auxiliary labeling.
To start, draw $P_{2n}$ in two rows as a bipartite graph. The second coordinates of the top-row vertices will be $0,1,\ldots,k-1,0,1,\ldots,k-1$ from left to right. The second coordinates of the bottom-row vertices will also be $0,1,\ldots,k-1,0,1,\ldots,k-1$ from left to right. 
The first coordinates of the top-row alternate between $0$ and $1$, starting with $1$, while those of the bottom row alternate starting with $0$. 

It is clear that this labeling uses each element of $A$ as a vertex label exactly twice. However, the labeling constructed so far is not $A$-cordial.  The second coordinates of the edge labels cycle through $0,1,\ldots,k-1,0,1,\ldots,k-1$ in order four times, while the first coordinates of the edge labels alternate between $0$ and $1$. Since $k$ is even, this means that each $(i,j) \in \Z_2 \times \Z_k$ with $i+j$ odd appears as a label on four edges, while each $(i,j) \in \Z_2 \times \Z_k$ with $i+j$ even does not appear as an edge label. 

The final step is to modify this labeling by swapping (the first coordinates of) the last $m$ labels of the top row with (the first coordinates of) the second batch of $m$ labels from the bottom row. This swapping is illustrated in Figure~\ref{fig:Z2xZ4} for the example $A = \Z_2 \times \Z_4$. 

\begin{figure}[h]
    \centering
    \begin{tikzpicture}
    \node (1) at (0,0) {$10$};
     \node (3) at (2,0) {$01$};
      \node (5) at (4,0) {$12$};
          \node (7) at (6,0) {$03$};
     \node (9) at (8,0) {$10$};
      \node (11) at (10,0) {$01$};
      \node (13) at (12,0) {\fbox{$12$}};
      \node (15) at (14,0) {\fbox{$03$}};
          \node (2) at (1,-1.5) {$00$};
     \node (4) at (3,-1.5) {$11$};
      \node (6) at (5,-1.5) {\fbox{$02$}};
          \node (8) at (7,-1.5) {\fbox{$13$}};
     \node (10) at (9,-1.5) {$00$};
      \node (12) at (11,-1.5) {$11$};
       \node (14) at (13,-1.5) {$02$};
        \node (16) at (15,-1.5) {$13$};
         \draw (1) -- (2) node [midway,below left] {\tiny{$\red{10}$}};
         \draw (2) -- (3) node [midway,above left] {\tiny{$\red{01}$}};
         \draw (3) -- (4) node [midway,below left] {\tiny{$\red{12}$}};
         \draw (4) -- (5) node [midway,above left] {\tiny{$\red{03}$}};
         \draw (5) -- (6) node [midway,below left] {\tiny{$\red{10}$}};
         \draw (6) -- (7) node [midway,above left] {\tiny{$\red{01}$}};
         \draw (7) -- (8) node [midway,below left] {\tiny{$\red{12}$}};
         \draw (8) -- (9) node [midway,above left] {\tiny{$\red{03}$}};
         \draw (9) -- (10) node [midway,below left] {\tiny{$\red{10}$}};
         \draw (10) -- (11) node [midway,above left] {\tiny{$\red{01}$}};
         \draw (11) -- (12) node [midway,below left] {\tiny{$\red{12}$}};
         \draw (12) -- (13) node [midway,above left] {\tiny{$\red{03}$}};
         \draw (13) -- (14) node [midway,below left] {\tiny{$\red{10}$}};
         \draw (14) -- (15) node [midway,above left] {\tiny{$\red{01}$}};
         \draw (15) -- (16) node [midway,below left] {\tiny{$\red{12}$}};
    \end{tikzpicture}\\
    \vspace{.1in}
 $\bigg\downarrow$\\
    \vspace{.1in}
      \begin{tikzpicture}
    \node (1) at (0,0) {$10$};
     \node (3) at (2,0) {$01$};
      \node (5) at (4,0) {$12$};
          \node (7) at (6,0) {$03$};
     \node (9) at (8,0) {$10$};
      \node (11) at (10,0) {$01$};
      \node (13) at (12,0) {\fbox{$02$}};
      \node (15) at (14,0) {\fbox{$13$}};
          \node (2) at (1,-1.5) {$00$};
     \node (4) at (3,-1.5) {$11$};
      \node (6) at (5,-1.5) {\fbox{$12$}};
          \node (8) at (7,-1.5) {\fbox{$03$}};
     \node (10) at (9,-1.5) {$00$};
      \node (12) at (11,-1.5) {$11$};
       \node (14) at (13,-1.5) {$02$};
        \node (16) at (15,-1.5) {$13$};
         \draw (1) -- (2) node [midway,below left] {$\red{10}$};
         \draw (2) -- (3) node [midway,above left] {$\red{01}$};
         \draw (3) -- (4) node [midway,below left] {\tiny{$\red{12}$}};
         \draw (4) -- (5) node [midway,above left] {\tiny{$\red{03}$}};
         \draw (5) -- (6) node [midway,below left] {\fbox{\tiny{$\red{00}$}}};
         \draw (6) -- (7) node [midway,above left] {\fbox{\tiny{$\red{11}$}}};
         \draw (7) -- (8) node [midway,below left] {\fbox{\tiny{$\red{02}$}}};
         \draw (8) -- (9) node [midway,above left] {\fbox{\tiny{$\red{13}$}}};
         \draw (9) -- (10) node [midway,below left] {\tiny{$\red{10}$}};
         \draw (10) -- (11) node [midway,above left] {\tiny{$\red{01}$}};
         \draw (11) -- (12) node [midway,below left] {\tiny{$\red{12}$}};
         \draw (12) -- (13) node [midway,above left] {\fbox{\tiny{$\red{13}$}}};
         \draw (13) -- (14) node [midway,below left] {\fbox{\tiny{$\red{00}$}}};
         \draw (14) -- (15) node [midway,above left] {\fbox{\tiny{$\red{11}$}}};
         \draw (15) -- (16) node [midway,below left] {\fbox{\tiny{$\red{02}$}}};
    \end{tikzpicture}
    \caption{An example of the labeling process from the proof of Proposition~\ref{prop:doublepath}, as illustrated for $A=\Z_2\times\Z_4$. (Note that this example is only for illustrative purposes, as we already know from Theorem~\ref{prop:smallorder} that this $A$ is $\mathbb{P}$-cordial.) The upper graph shows the auxiliary labeling with boxes around the vertex labels to be swapped. The lower graph shows the $A$-cordial labeling obtained by swapping, with boxes around the (vertex and edge) labels that have changed.}
    \label{fig:Z2xZ4}
\end{figure}
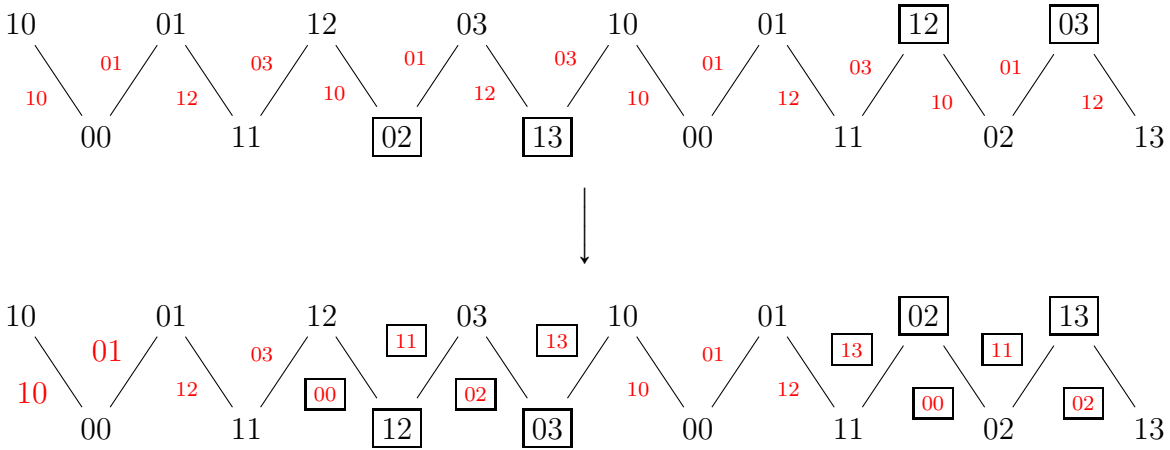

It is clear that this modified labeling also uses each element of $A$ as a vertex label exactly twice. The second coordinates of the edge labels still cycle through $0,1,\ldots,k-1,0,1,\ldots,k-1$ in order four times. However, the pattern of first coordinates of edge labels is now more complicated. 

Edge labels appear in the following sequence from left to right (we include line breaks that we believe help clarify the structure of the sequence):
  \begin{align*}
   &10, 01, 12, 03, \ldots, 1(k-2), 0(k-1),\\
 &00, 11, 02, 13, \ldots, 0(k-2), 1(k-1),\\
 &10, 01, 12, 03, \ldots,1(k-2), \\
 &1(k-1),\\
  &00, 11, 02, 13, \ldots, 0(k-2).
  \end{align*}

 In particular, the first $k$ labels are exactly the same as the second batch of $k$ labels and in the same order, except that all the first coordinates have been switched. Hence, the first $n$ edge labels consist of each element of $A$ exactly once. Similarly, after the first $n$ edge labels, the next $k-1$ labels are the same as the last $k-1$ edge labels and in the same order, except that all the first coordinates have been switched.
 Hence we see that, in total, each label appears exactly twice as an edge label, except for the label $0(k-1)$, which appears exactly once. Thus, this is an $A$-cordial labeling of $P_{2n}$.
\end{proof}

\section*{Acknowledgements}
	OP was supported by a Mathematical Sciences Postdoctoral Research Fellowship (\#1703696) from the National Science Foundation.


\bibliographystyle{alpha}
\bibliography{cordial}

\end{document}